\documentclass[11pt]{amsart}

\usepackage[margin=1in]{geometry}
\usepackage[T1]{fontenc}
\usepackage[utf8]{inputenc}
\usepackage{amsmath,amssymb,amsthm}
\usepackage{hyperref}

\numberwithin{equation}{section}

\usepackage{graphicx,color}

\newtheorem{theorem}{Theorem}[section]
\newtheorem{lemma}[theorem]{Lemma}

\theoremstyle{definition}

\theoremstyle{remark}

\newcommand{\Sone}{\mathbb S^1}
\newcommand{\C}{\mathbb C}
\newcommand{\Z}{\mathbb Z}

\newcommand{\VMO}{\operatorname{VMO}}
\newcommand{\BMO}{\operatorname{BMO}}

\newcommand{\ind}{\mathbf 1}

\title[Degenerate constants: problems of Brezis]{Degenerate constants in degree inequalities for Sobolev circle maps: on some problems posed by Brezis}

\author{Xu'an Dou}\address{Beijing International Center for Mathematical Research, Peking University, Beijing, 100871, China.}\email{dxa@pku.edu.cn}

\author{Zeyu Jin}\address{School of Mathematical Sciences, Peking University, Beijing, 100871, China.}\email{jinzy@pku.edu.cn}

\date{\today}

\begin{document}

\begin{abstract}

We sharpen the constants in two degree inequalities for circle-valued Sobolev
maps in degenerate regimes, as \(p\to1^+\) or \(\delta\to0^+\). The two proofs
use the same power trick together with elementary estimates. The results answer
two open problems posed by Brezis \cite[Section 5]{Brezis2023Favorite}. The proofs are obtained by generative AI and are verified by the authors.

\end{abstract}

\maketitle

\section{Introduction}

Let
\[
\Sone=\{z\in\C: |z|=1\}.
\]
We identify \(\Sone\) with the unit circle in \(\C\), equip it with arclength
measure, and write \(|x-y|\) for the Euclidean chordal distance. For a continuous map \(f:\Sone\to\Sone\), \(\deg f\in\Z\) denotes the
usual winding number. We shall also use the VMO degree of Brezis and Nirenberg:
if \(f\in\VMO(\Sone;\Sone)\), then \(\deg f\) is well-defined, agrees with the
classical degree for continuous maps, and is continuous with respect to
\(\BMO\)-convergence \cite{BrezisNirenberg1995}. Here \(u\in\VMO(\Sone)\) means
that
\[
\lim_{r\downarrow0}\sup_{|I|\le r}\frac1{|I|}\int_I |u-u_I|\,\mathrm{d} x=0,
\qquad
u_I:=\frac1{|I|}\int_I u\,\mathrm{d} x,
\]
where the supremum is taken over arcs \(I\subset\Sone\). We refer to
\cite[Chapter 12]{BrezisMironescu2021Sobolev} for a systematic treatment of the
VMO degree for circle-valued maps.

The present paper concerns two estimates for this degree. The first lies in the
critical fractional scale \(W^{1/p,p}(\Sone;\Sone)\), which is contained in
\(\VMO(\Sone;\Sone)\) for \(1<p<\infty\). The second is formulated for continuous
circle-valued maps in terms of a nonlocal functional with a threshold $\delta>0$.  Our goal is to sharpen the constants in the degenerate regimes, as \(p\to1^+\) or \(\delta\to0^+\). 

We first recall the two estimates. The following estimate is the \(p=2\) case of \cite[Theorem 5.1]{Brezis2023Favorite}; see also
\cite[Corollary 0.5 and formula (0.7)]{BourgainBrezisMironescu2005}.  
\begin{theorem}[$H^{1/2}$ Degree estimate]\label{thm:hhalf}
There exists a universal constant \(C_0\) such that every
\(g\in H^{1/2}(\Sone;\Sone)\) satisfies
\[
 |\deg g|\le C_0
 \int_{\Sone}\int_{\Sone}
 \frac{|g(x)-g(y)|^2}{|x-y|^2}\,\mathrm{d} x \, \mathrm{d} y .
\]
\end{theorem}

The integral in Theorem \ref{thm:hhalf} is equivalent to the square of the
\(H^{1/2}\)-seminorm. Since \(|z-w|^2\le 2^{2-p}|z-w|^p\) for \(z,w\in\Sone\)
and \(1<p\le2\), one obtains the classical estimate
\[
 |\deg f|\le C_p
 \int_{\Sone}\int_{\Sone}
 \frac{|f(x)-f(y)|^p}{|x-y|^2}\,\mathrm{d} x \, \mathrm{d} y,
 \qquad f\in W^{1/p,p}(\Sone;\Sone),
\]
with a constant \(C_p\) depending on \(p\). In the form recorded by Brezis, this
estimate holds for every \(1<p<\infty\) \cite[Theorem 5.1]{Brezis2023Favorite}.
As observed there, however, the right-hand side diverges as \(p\downarrow1\)
unless \(f\) is constant. It is therefore natural to ask whether the constant can
be chosen with the compensating factor \(p-1\).

We next recall the threshold estimate. For \(f\in C(\Sone;\Sone)\) and
\(\delta>0\), define
\begin{equation}\label{eq:Idelta}
I_\delta(f)=
\int_{\Sone}\int_{\Sone}
\frac{\ind_{\{|f(x)-f(y)|\ge\delta\}}}{|x-y|^2}\,\mathrm{d} x \, \mathrm{d} y .
\end{equation}

The following theorem is the fixed-threshold estimate proved by Nguyen \cite{Nguyen2007} and recorded as \cite[Theorem 5.2]{Brezis2023Favorite}.

\begin{theorem}[Nguyen's optimal threshold estimate \cite{Nguyen2007}; cf. \cite{Brezis2023Favorite}]\label{thm:fixed-sqrt3}
There exists a universal constant \(C_1\) such that every continuous map
\(g:\Sone\to\Sone\) satisfies
\[
 |\deg g|\le C_1 I_{\sqrt3}(g).
\]
\end{theorem}

The threshold \(\sqrt3\) is optimal in Theorem \ref{thm:fixed-sqrt3}; see
\cite{Nguyen2007}. Since \(I_\delta(f)\) is nonincreasing as \(\delta\) increases,
Theorem \ref{thm:fixed-sqrt3} implies
\[
 |\deg f|\le C_1 I_\delta(f),
 \qquad 0<\delta\le\sqrt3.
\]
This bound deteriorates as \(\delta\downarrow0\), because \(I_\delta(f)\) tends to
infinity for nonconstant continuous maps. Brezis therefore asked whether the
constant can be replaced by one of order \(\delta\). Both the factor $\delta$ here and the previous factor $p-1$ were predicted in \cite{Brezis2023Favorite} using a connection to the Cauchy formula.

The two main results of the paper answer these questions affirmatively. 

\begin{theorem}[Endpoint nonlocal degree estimate; Open Problem 5.1 of \cite{Brezis2023Favorite}]\label{thm:op51}
There exists a universal constant \(c\) such that, for every \(1<p\le2\) and
every \(f\in W^{1/p,p}(\Sone;\Sone)\),
\[
|\deg f|\le c(p-1)
\int_{\Sone}\int_{\Sone}
\frac{|f(x)-f(y)|^p}{|x-y|^2}\,\mathrm{d} x \, \mathrm{d} y .
\]
\end{theorem}

\begin{theorem}[Small-threshold degree estimate; Open Problem 5.3 of \cite{Brezis2023Favorite}]\label{thm:op53}
There exists a universal constant \(c\) such that, for every continuous
\(f:\Sone\to\Sone\) and every \(0<\delta\le\sqrt3\),
\[
|\deg f|\le c\delta I_\delta(f).
\]
\end{theorem}

Theorem \ref{thm:op51} answers \cite[Open Problem 5.1]{Brezis2023Favorite},
which goes back to \cite[Remark 7]{Brezis2006NewQuestions}. Theorem
\ref{thm:op53} answers \cite[Open Problem 5.3]{Brezis2023Favorite}, previously
posed in \cite[Open Problem 3]{Brezis2006NewQuestions}. An analogue of Theorem
\ref{thm:op53} for maps \(\mathbb S^N\to\mathbb S^N\) with \(N\ge2\) was proved
by Nguyen \cite{Nguyen2017RefinedDegree}; the circle case is the one treated here.

The proofs use a power identity of the degree: if \(f\in\VMO(\Sone;\Sone)\) and \(k\ge1\) is an integer, then
\(f^k\in\VMO(\Sone;\Sone)\) and
\begin{equation}\label{power-fk}
    \deg(f^k)=k\deg f.
\end{equation} This is a basic property of the degree and follows from e.g. $\deg fg= \deg f+ \deg g$ for two VMO maps $f,g$, see \cite[Theorem 12.1]{BrezisMironescu2021Sobolev}. It can also be viewed as a functorial identity, since \(f^k=P_k\circ f\), where
\(P_k:\Sone\to\Sone\), \(P_k(z)=z^k\) and  $\deg f^k=\deg(P_k)\deg f=k\deg f$.

Both proofs use the same power trick. In the proof of Theorem \ref{thm:op51} we
apply Theorem \ref{thm:hhalf} to the powers \(f^k\) and average over \(k\) with
weights chosen so that the endpoint factor \(p-1\) appears. In the proof of
Theorem \ref{thm:op53} we choose a single power \(n\simeq\delta^{-1}\); the map
\(f^n\) converts the threshold \(\delta\) into the fixed threshold \(\sqrt3\), and
\(\deg(f^n)=n\deg f\) yields the extra factor \(\delta\).

The proofs of Theorems \ref{thm:op51} and \ref{thm:op53} were obtained by the generative AI system Rethlas \cite{JuEtAl2026Rethlas}\footnote{The raw output of
Rethlas is available at
\href{https://github.com/frenzymath/Rethlas_results/tree/main/analysis/Brezis_OpenProblems}{Rethlas Results Homepage}.} and were subsequently checked and
verified by the authors.

The remainder of the paper is organized as follows. Section 2 proves Theorem
\ref{thm:op51}. Section 3 proves Theorem \ref{thm:op53}.

\section{Endpoint Nonlocal Degree Estimate}

This section proves Theorem \ref{thm:op51}.

\begin{lemma}\label{lem:hhalf-powers}
There is a universal constant \(C_0\) such that, whenever
\(f\in\VMO(\Sone;\Sone)\), \(k\ge1\), and \(f^k\in H^{1/2}(\Sone;\Sone)\),
\[
 k|\deg f|\le C_0
 \int_{\Sone}\int_{\Sone}
 \frac{|f(x)^k-f(y)^k|^2}{|x-y|^2}\,\mathrm{d} x \, \mathrm{d} y .
\]
\end{lemma}

\begin{proof}
Apply Theorem \ref{thm:hhalf} to \(g=f^k\) and use
\(\deg(f^k)=k\deg f\).
\end{proof}
The main technical lemma is below, which is for the case $1<p\leq 3/2$.
\begin{lemma}\label{lem:weighted-powers}
There is a universal constant \(C_2\) such that for every \(1<p\le 3/2\) one can
choose positive weights
\[
a_k=\frac{k^{-p-1}}{\sum_{j=1}^\infty j^{-p}},\qquad k\ge1,
\]
for which
\[
\sum_{k=1}^\infty a_k k=1
\]
and, for every \(z,w\in\Sone\),
\[
\sum_{k=1}^\infty a_k |z^k-w^k|^2
\le C_2 (p-1)|z-w|^p .
\]
\end{lemma}

\begin{proof}
The normalization gives
\[
\sum_{k=1}^\infty a_k k
=\frac{\sum_{k=1}^\infty k^{-p}}{\sum_{j=1}^\infty j^{-p}}
=1.
\]
Also
\[
\sum_{j=1}^\infty j^{-p}\ge \int_1^\infty t^{-p}\,dt=\frac1{p-1},
\]
so the reciprocal of the denominator is at most \(p-1\).

Put \(\rho=|z-w|\). Since \(|z|=|w|=1\),
\[
|z^k-w^k|\le k|z-w|=k\rho
\quad\text{and}\quad
|z^k-w^k|\le 2.
\]
Thus
\[
\sum_{k=1}^\infty k^{-p-1}|z^k-w^k|^2
\le
\sum_{k=1}^\infty k^{-p-1}\min\{k^2\rho^2,4\}.
\]
If \(\rho\ge1\), the right hand side is bounded by a universal constant, hence
by \(C\rho^p\). If \(0<\rho<1\), split the sum at \(K=\lfloor \rho^{-1}\rfloor\).
Then, since \(1<p\le3/2\),
\[
\rho^2\sum_{k\le K} k^{1-p}
\le C\rho^2 K^{2-p}
\le C\rho^p
\]
and
\[
\sum_{k>K} k^{-p-1}\le C K^{-p}\le C\rho^p.
\]
The case \(\rho=0\) is trivial. Multiplying by the reciprocal of
\(\sum_{j=1}^\infty j^{-p}\), which is at most \(p-1\), proves the claim.
\end{proof}

The case $3/2\le p\le 2$ is straightforward  as the main focus is the behavior when $p\rightarrow1^+$.
\begin{lemma}\label{lem:large-p}
There is a universal constant \(C_3\) such that for every \(3/2\le p\le2\) and
every \(f\in W^{1/p,p}(\Sone;\Sone)\),
\[
|\deg f|\le C_3(p-1)
\int_{\Sone}\int_{\Sone}
\frac{|f(x)-f(y)|^p}{|x-y|^2}\,\mathrm{d} x \, \mathrm{d} y .
\]
\end{lemma}

\begin{proof}
For \(z,w\in\Sone\) and \(p\le2\),
\[
|z-w|^2\le 2^{2-p}|z-w|^p.
\]
Therefore \(f\in H^{1/2}(\Sone;\Sone)\) and
\[
\int_{\Sone}\int_{\Sone}
\frac{|f(x)-f(y)|^2}{|x-y|^2}\,\mathrm{d} x \, \mathrm{d} y
\le
2^{2-p}
\int_{\Sone}\int_{\Sone}
\frac{|f(x)-f(y)|^p}{|x-y|^2}\,\mathrm{d} x \, \mathrm{d} y .
\]
Applying Theorem \ref{thm:hhalf} and using \(p-1\ge1/2\) gives the assertion
after increasing the universal constant.
\end{proof}

\begin{lemma}\label{lem:small-p}
There is a universal constant \(C_4\) such that for every \(1<p\le3/2\) and
every \(f\in W^{1/p,p}(\Sone;\Sone)\),
\[
|\deg f|\le C_4(p-1)
\int_{\Sone}\int_{\Sone}
\frac{|f(x)-f(y)|^p}{|x-y|^2}\,\mathrm{d} x \, \mathrm{d} y .
\]
\end{lemma}

\begin{proof}
Let \(a_k\) be the weights from Lemma \ref{lem:weighted-powers}. For each
\(k\ge1\), the positivity of \(a_k\) and Lemma \ref{lem:weighted-powers} give
the pointwise estimate
\[
a_k |f(x)^k-f(y)^k|^2
\le
C_2(p-1)|f(x)-f(y)|^p .
\]
After division by \(|x-y|^2\) and integration,
\[
a_k
\int_{\Sone}\int_{\Sone}
\frac{|f(x)^k-f(y)^k|^2}{|x-y|^2}\,\mathrm{d} x \, \mathrm{d} y
\le
C_2(p-1)
\int_{\Sone}\int_{\Sone}
\frac{|f(x)-f(y)|^p}{|x-y|^2}\,\mathrm{d} x \, \mathrm{d} y
<\infty .
\]
Since \(a_k>0\), this shows \(f^k\in H^{1/2}(\Sone;\Sone)\). Hence Lemma
\ref{lem:hhalf-powers} applies, and
\[
k|\deg f|
\le C_0
\int_{\Sone}\int_{\Sone}
\frac{|f(x)^k-f(y)^k|^2}{|x-y|^2}\,\mathrm{d} x \, \mathrm{d} y .
\]
Multiplying by \(a_k\), summing in \(k\), and using
\(\sum_k a_k k=1\), Tonelli's theorem, and Lemma
\ref{lem:weighted-powers}, we get
\[
\begin{aligned}
|\deg f|
&\le C_0
\int_{\Sone}\int_{\Sone}
\frac{\sum_{k=1}^\infty a_k |f(x)^k-f(y)^k|^2}{|x-y|^2}\,\mathrm{d} x \, \mathrm{d} y\\
&\le C_0C_2(p-1)
\int_{\Sone}\int_{\Sone}
\frac{|f(x)-f(y)|^p}{|x-y|^2}\,\mathrm{d} x \, \mathrm{d} y .
\end{aligned}
\]
This is the desired estimate.
\end{proof}

\begin{proof}[Proof of Theorem \ref{thm:op51}]
If \(3/2\le p\le2\), the estimate follows from Lemma
\ref{lem:large-p}. If \(1<p\le3/2\), it follows from Lemma
\ref{lem:small-p}. Taking
\[
c=\max\{C_3,C_4\}
\]
gives a universal constant valid for every \(1<p\le2\) and every
\(f\in W^{1/p,p}(\Sone;\Sone)\).
\end{proof}

\section{Threshold Integral Estimates}

This section proves Theorem \ref{thm:op53}. We first give a lemma which allows to connect $I_{\delta}(f)$ with $I_{\sqrt{3}}(f^n)$ with a suitable $n$.

\begin{lemma}\label{lem:power-threshold-inclusion}
Let \(0<\delta\le\sqrt3\), and put
\[
\alpha=2\arcsin(\delta/2),\qquad \alpha_0=2\pi/3.
\]
Choose
\[
n=\left\lfloor \frac{\alpha_0}{\alpha}\right\rfloor .
\]
Then \(n\ge1\), \(n\alpha\le\alpha_0\), and
\[
\frac1n\le \frac{2\alpha}{\alpha_0}\le C_5\delta
\]
for a universal constant \(C_5\). Moreover, for any \(a,b\in\Sone\),
\[
|a^n-b^n|\ge\sqrt3\quad\Longrightarrow\quad |a-b|\ge\delta .
\]
\end{lemma}

\begin{proof}
Since \(0<\delta\le\sqrt3\), the corresponding angular threshold satisfies
\[
0<\alpha=2\arcsin(\delta/2)\le 2\arcsin(\sqrt3/2)=2\pi/3=\alpha_0,
\]
so \(n\ge1\) and \(n\alpha\le\alpha_0\).

If \(\alpha_0/\alpha<2\), then \(n=1\), hence
\[
\frac1n=1<\frac{2\alpha}{\alpha_0}.
\]
If \(\alpha_0/\alpha\ge2\), then
\[
n=\left\lfloor \frac{\alpha_0}{\alpha}\right\rfloor
\ge \frac12\,\frac{\alpha_0}{\alpha},
\]
and again \(1/n\le 2\alpha/\alpha_0\). Finally,
\[
\alpha=2\arcsin(\delta/2)\le \frac{2\pi}{3\sqrt3}\,\delta
\]
on \(0<\delta\le\sqrt3\), so \(1/n\le C_5\delta\) for a universal \(C_5\).

For the separation implication, let \(\theta\in[0,\pi]\) be the angular distance
between \(a\) and \(b\). Equivalently, choose arguments for \(a\) and \(b\) whose
difference has absolute value \(\theta\). Then \(|a-b|=2\sin(\theta/2)\). If
\(|a-b|<\delta\), then \(\theta<\alpha\). Hence
\[
n\theta<n\alpha\le\alpha_0<\pi.
\]
Since \(n\theta<\pi\), multiplication by \(n\) does not wrap the chosen angular
separation around modulo \(2\pi\). Thus the angular distance between \(a^n\) and
\(b^n\) is exactly \(n\theta\), and therefore
\[
|a^n-b^n|=2\sin(n\theta/2)<2\sin(\alpha_0/2)=\sqrt3.
\]
Taking the contrapositive gives the claim.
\end{proof}

\begin{lemma}\label{lem:power-energy-comparison}
Let \(f:\Sone\to\Sone\) be continuous and let \(n\) be as in Lemma
\ref{lem:power-threshold-inclusion}. Define \(g:\Sone\to\Sone\) by
\[
g(x)=f(x)^n.
\]
Then
\[
\deg g=n\deg f
\]
and
\[
I_{\sqrt{3}}(g)
\le
I_\delta(f).
\]
\end{lemma}

\begin{proof}
The degree identity is the same as \eqref{power-fk}.

By Lemma \ref{lem:power-threshold-inclusion}, the set
\[
\{(x,y): |g(x)-g(y)|\ge\sqrt3\}
\]
is contained in
\[
\{(x,y): |f(x)-f(y)|\ge\delta\}.
\]
The kernels are identical \eqref{eq:Idelta}, so integration over the smaller set gives the stated
energy comparison.
\end{proof}

\begin{proof}[Proof of Theorem \ref{thm:op53}]
Let \(0<\delta\le\sqrt3\), set \(\alpha=2\arcsin(\delta/2)\), and choose
\[
n=\left\lfloor\frac{2\pi/3}{\alpha}\right\rfloor .
\]
Let \(g=f^n\). By Theorem \ref{thm:fixed-sqrt3} and Lemma
\ref{lem:power-energy-comparison},
\[
n|\deg f|
=|\deg g|
\le
C_1
I_{\sqrt{3}}(g)
\le
C_1 I_\delta(f).
\]
Therefore
\[
|\deg f|\le \frac{C_1}{n}I_\delta(f).
\]
Lemma \ref{lem:power-threshold-inclusion} gives \(1/n\le C_5\delta\). Hence
\[
|\deg f|\le C_1C_5\,\delta\, I_\delta(f).
\]
Thus the desired estimate holds with the universal constant \(c=C_1C_5\).
\end{proof}

\section*{Acknowledgements}

The authors would like to thank the Rethlas team, namely Haocheng Ju, Jiedong
Jiang, Shurui Liu, Guoxiong Gao, Yuefeng Wang, Zeming Sun, Leheng Chen, Bin Wu,
Liang Xiao, and Bin Dong, for their contributions to the development of Rethlas \cite{JuEtAl2026Rethlas}.

\providecommand{\bysame}{\leavevmode\hbox to3em{\hrulefill}\thinspace}
\providecommand{\MR}{\relax\ifhmode\unskip\space\fi MR }
\providecommand{\MRhref}[2]{%
  \href{http://www.ams.org/mathscinet-getitem?mr=#1}{#2}
}
\providecommand{\href}[2]{#2}


\begin{thebibliography}{1}

\bibitem{BourgainBrezisMironescu2005}
Jean Bourgain, Ha{\"i}m Brezis, and Petru Mironescu, \emph{Lifting, degree, and
  distributional {J}acobian revisited}, Communications on Pure and Applied
  Mathematics \textbf{58} (2005), no.~4, 529--551.

\bibitem{Brezis2006NewQuestions}
Ha{\"i}m Brezis, \emph{New questions related to the topological degree}, The
  Unity of Mathematics, Progress in Mathematics, vol. 244, Birkh{\"a}user
  Boston, Boston, MA, 2006, pp.~137--154.

\bibitem{Brezis2023Favorite}
\bysame, \emph{Some of my favorite open problems}, Atti Accad. Naz. Lincei Cl.
  Sci. Fis. Mat. Natur. \textbf{34} (2023), no.~2, 307--335.

\bibitem{BrezisMironescu2021Sobolev}
Ha{\"i}m Brezis and Petru Mironescu, \emph{Sobolev maps to the circle: From the
  perspective of analysis, geometry, and topology}, Progress in Nonlinear
  Differential Equations and Their Applications, vol.~96, Birkh{\"a}user, New
  York, NY, 2021.

\bibitem{BrezisNirenberg1995}
Ha{\"i}m Brezis and Louis Nirenberg, \emph{Degree theory and {BMO}; part {I}:
  Compact manifolds without boundaries}, Selecta Mathematica, New Series
  \textbf{1} (1995), no.~2, 197--263.

\bibitem{JuEtAl2026Rethlas}
Haocheng Ju, Guoxiong Gao, Jiedong Jiang, Bin Wu, Zeming Sun, Leheng Chen,
  Yutong Wang, Yuefeng Wang, Zichen Wang, Wanyi He, et~al., \emph{Automated
  conjecture resolution with formal verification}, arXiv preprint
  arXiv:2604.03789 (2026), 1--28.

\bibitem{Nguyen2007}
Hoai-Minh Nguyen, \emph{Optimal constant in a new estimate for the degree},
  Journal d'Analyse Mathematique \textbf{101} (2007), 367--395.

\bibitem{Nguyen2017RefinedDegree}
\bysame, \emph{A refined estimate for the topological degree}, Comptes Rendus.
  Math\'{e}matique \textbf{355} (2017), no.~10, 1046--1049.

\end{thebibliography}
\end{document}